\documentclass{amsart}
\numberwithin{equation}{section}

\usepackage{bbm}
\usepackage{comment}
\usepackage{verbatim}
\usepackage{mathrsfs}
\usepackage{latexsym}
\usepackage{epsfig}
\usepackage{amsmath}
\usepackage[usenames,dvipsnames]{xcolor}
\usepackage{bbm}
\usepackage{graphics}
\usepackage{amssymb}
\usepackage{amsthm}
\usepackage[alphabetic]{amsrefs}
\usepackage{amsopn}
\usepackage{amscd}
\usepackage[all,knot]{xy}
\xyoption{all}
\usepackage{rotating}
\usepackage{tikz}
\usetikzlibrary{calc}
\usepgflibrary{shapes.geometric}
\usepgflibrary{shapes.misc}
\usetikzlibrary{positioning}
\usetikzlibrary{decorations}
\usetikzlibrary{decorations.pathreplacing}
\usepackage{hyperref}

\theoremstyle{plain}
\newtheorem{thm}{Theorem}[section]
\newtheorem{lem}[thm]{Lemma}
\newtheorem{prop}[thm]{Proposition}
\newtheorem{cor}[thm]{Corollary}

\newtheorem*{thm*}{Theorem}
\newtheorem*{lem*}{Lemma}
\newtheorem*{prop*}{Proposition}
\newtheorem*{cor*}{Corollary}

\theoremstyle{definition}
\newtheorem{defn}[thm]{Definition}
\newtheorem*{defn*}{Definition}

\newtheorem{ex}[thm]{Example}
{}
\newtheorem{rem}[thm]{Remark}
\newtheorem*{rem*}{Remark}

{}
{}
{}
\newtheorem*{ack}{Acknowledgements}{}

\theoremstyle{remark}
{}
{}
{}

\def\to{\longrightarrow}

\def\ZZ{\mathbb{Z}}

\def\A{\mathcal{A}}
\def\Ind{\mathrm{Ind}}

\def\sat{\mathrm{sat}}

\def\sfA{\mathsf{A}}
\def\sfD{\mathsf{D}}
\def\sfK{\mathsf{K}}
\def\sfT{\mathsf{T}}

\def\sfP{\mathsf{P}}
\def\sfM{\mathsf{M}}

\def\sfC{\mathsf{C}}

\def\mcS{\mathcal{S}}

\def\mcF{\mathcal{F}}

\def\base{\mathbbm{k}}

\DeclareMathOperator{\End}{End}
\DeclareMathOperator{\Spec}{Spec}

\DeclareMathOperator{\coker}{coker}

\DeclareMathOperator{\im}{im}
\DeclareMathOperator{\Hom}{Hom}
\DeclareMathOperator{\RHom}{\mathbf{R}Hom}
\DeclareMathOperator{\RuHom}{\mathbf{R}\underline{Hom}}

\DeclareMathOperator{\Ext}{Ext}

\DeclareMathOperator{\cone}{cone}

\DeclareMathOperator{\modu}{\mathsf{mod}}
\DeclareMathOperator{\Modu}{\mathsf{Mod}}

\DeclareMathOperator{\Gr}{Gr}
\DeclareMathOperator{\gr}{gr}

\DeclareMathOperator{\Inj}{Inj}

\DeclareMathOperator{\add}{add}

\DeclareMathOperator{\thick}{thick}

\DeclareMathOperator{\Thick}{\mathrm{Thick}}

\DeclareMathOperator{\Add}{Add}
\DeclareMathOperator{\soc}{soc}
\DeclareMathOperator{\length}{length}
\DeclareMathOperator{\Sat}{Sat}

\DeclareMathOperator{\NC}{NC}

\title[The dual numbers, arcs, and non-crossing partitions]{On the graded dual numbers, arcs, and non-crossing partitions of the integers}
\author{Sira Gratz}
\address{Sira Gratz, 
University of Oxford,
Mathematical Institute,
Andrew Wiles Building,
Radcliffe Observatory Quarter,
Woodstock Road,
Oxford,
OX2 6GG}
\email{sira.gratz@maths.ox.ac.uk}
\author{Greg Stevenson}
\address{Greg Stevenson, Universit\"at Bielefeld, Fakult\"at f\"ur Mathematik, BIREP Gruppe, Postfach 10\,01\,31, 33501 Bielefeld, Germany.}
\email{gstevens@math.uni-bielefeld.de}

\begin{document}

\begin{abstract}
\noindent We give a combinatorial model for the bounded derived category of graded modules over the dual numbers in terms of arcs on the integer line with a point at infinity. Using this model we describe the lattice of thick subcategories of the bounded derived category, and of the perfect complexes, in terms of non-crossing partitions. We also make some comments on the symmetries of these lattices, exceptional collections, and the analogous problem for the unbounded derived category.
\end{abstract}

\maketitle

\setcounter{tocdepth}{1}
\tableofcontents

\section{Introduction}

There is, by now, a small industry of classifying well-behaved subcategories of triangulated categories. Here well-behaved usually means, at a minimum, closed under suspension, cones, and summands but frequently also entails being closed under some additional operation, such as a tensor product if the triangulated category is monoidal. This program of understanding the coarse structure of triangulated categories has its genesis in the pioneering work of Devinatz, Hopkins, and Smith \cite{DevinatzHopkinsSmith}, Neeman \cite{NeeChro}, and Benson, Carlson, and Rickard \cite{BCR}, and has led to a great deal of progress in algebra, geometry, and topology.

The aim of this article is to contribute another piece to the puzzle and advertise a direction that is in need of further investigation. In some sense, following the blueprint provided by the classification theorems mentioned above and culminating in the work of Balmer \cite{BaSpec}, we now have a very good understanding of thick tensor ideals in essentially small tensor triangulated categories. One gets a `continuous' classification in terms of nice subsets of an associated topological space; all three of the works mentioned above fit into this paradigm.

On the other hand, there are many triangulated categories, arising for instance from representations of finite dimensional algebras, which do not necessarily naturally carry a reasonable monoidal structure. Our understanding of these cases is still extremely limited and there are few general techniques for describing the associated lattices of thick subcategories. However, certain examples are, at least partially, understood. Ingalls and Thomas \cite{IT} computed the lattice of thick subcategories for the bounded derived categories of simply laced Dynkin quivers (as well as in affine type). Here, the classification is purely combinatorial -- there is no `continuous part' -- and thick subcategories correspond to non-crossing partitions of the relevant Dynkin type.

More generally, work of Igusa, Schiffler, and Thomas \cite{IS} (cf.\ also Krause's paper \cite{KrauseReport}*{Theorem~6.10}) gives insight into the combinatorial aspect of the classification problem more generally. They give a classification of thick subcategories coming from exceptional collections in bounded derived categories of finite acyclic quivers via a more general version of non-crossing partitions using, as did Ingalls and Thomas, the Coxeter system corresponding to the quiver.

A very natural setting in which one observes both `continuous' and `combinatorial' aspects is the study of derived categories of graded modules over graded rings. In \cite{DSgraded} the continuous part of this classification is treated for graded commutative noetherian rings, the classification of thick tensor ideals being given in terms of certain subsets of the spectrum of homogeneous prime ideals. On the other hand, we are almost completely ignorant when it comes to the lattice of all thick subcategories in such examples.

In this paper we treat a small example, considering the lattice of thick subcategories for the bounded derived category of a graded exterior algebra on one generator (or equivalently, a graded polynomial ring on one generator). We give, in Corollary~\ref{cor_summary}, a complete description of the lattice of thick subcategories of $\sfD^\mathrm{b}(\gr \base[x]/(x^2))$ in terms of non-crossing partitions of the infinite poset $\ZZ\sqcup\{-\infty\}$, generalising the results of \cite{IT} and \cite{IS} to an infinite quiver. Our approach is based on exhibiting a geometric model for the bounded derived category which facilitates the computations of hom-sets and cones. We also discuss the action of various naturally occurring autoequivalences on the lattice of thick subcategories.

The paper is structured as follows: after giving some background and fixing notation in Section~\ref{sec_prelim} we give the geometric model for $\sfD^\mathrm{b}(\gr \base[x]/(x^2))$ in Section~\ref{sec_model}. In Section~\ref{sec_thick} we provide two different descriptions of the lattice of thick subcategories and discuss the action of certain autoequivalences on our model and the lattice. Finally, in Section~\ref{sec_loc}, we give some comments on the analogous problem for the unbounded derived category which we hope will serve as an invitation to others to work on this and related problems.

%-------------------------------------------------------------------------------------------------------------------------------------------------

%-------------------------------------------------------------------------------------------------------------------------------------------------

\section{Preliminaries and notation}\label{sec_prelim}

Throughout we will denote by $R$ the $\ZZ$-graded ring $\base[x]/(x^2)$ with $\vert x \vert = 1$ and by $S$ the $\ZZ$-graded ring $\base[y]$ with $\vert y \vert = 1$.

The purpose of this section is to collect some well-known facts we will use throughout and fix notation. Apart, perhaps, from some sign conventions our notation is relatively standard and this section could easily be skipped and then referred back to if warranted.

\subsection{Graded rings}
By a graded ring $A$ we will always mean a $\ZZ$-graded ring
\begin{displaymath}
A = \bigoplus_{i\in\ZZ} A_i.
\end{displaymath}
We will denote the category of all graded $A$-modules by $\Gr A$ and its full subcategory of finitely presented graded modules by $\gr A$. Throughout we will deal with honestly commutative graded rings so the reader may think of left or right modules as they prefer.

Given a graded $A$-module $M$ we write $M_i$ for the homogeneous elements of $M$ of degree $i$. The category of graded modules is equipped with a grading shift autoequivalence $(1)$ which acts on a graded module $M$ via
\begin{displaymath}
M(1)_i = M_{i+1}
\end{displaymath}
i.e.\ if $M$ is generated in degree $0$ then $M(1)$ is generated in degree $-1$. By taking powers and inverses we have an autoequivalence $(i)$ for every $i\in \ZZ$ and we refer to it as the $i$th \emph{twist} or \emph{internal degree shift}. It will sometimes be convenient to also use this notation to keep track of elements: if $m\in M_j$ we will write $m(i)$ for the corresponding element considered as an element of $M(i)$, now sitting in degree $j-i$.

As we will be using some Morita theory it is convenient to have an alternative description of a graded ring $A$ which is more categorical. The \emph{companion category} of $A$, denoted $\sfC_A$, is the category with objects labeled by $\ZZ$, morphisms
\begin{displaymath}
\sfC_A(i,j) = A_{j-i},
\end{displaymath}
and the obvious composition law and units. The category $\sfC_A$ is Morita equivalent to $A$ in the sense that there is an equivalence of categories
\begin{displaymath}
\Modu \sfC_A \cong \Gr A
\end{displaymath}
where $\Modu \sfC_A$ is the category of contravariant additive functors from $\sfC_A$ to abelian groups. The fact that $\sfC_A$ comes from a graded ring, and the grading shift on $\Gr A$ are reflected by the obvious translation automorphisms of $\sfC_A$. Further details can be found in \cite{DSgraded}. Two examples which are pertinent to this paper can be found in Example~\ref{ex_companion}.

Of particular interest to us are the graded rings $R = \base[x]/(x^2)$ and $S=\base[y]$ generated in degree $1$. We will make frequent use of the fact that $S$ is hereditary and $R$ is self-injective and graded representation finite:
\begin{displaymath}
\Gr R = \Add(R(i), \base(j)\; \vert \; i,j\in \ZZ),
\end{displaymath}
i.e.\ every graded $R$-module is a sum of free modules and simples. 

We denote the bounded derived category of finitely generated $R$-modules by $\sfD^\mathrm{b}(\gr R)$ and by $\sfD^\mathrm{perf}(\gr R)$ the perfect complexes, i.e.\ the full subcategory of complexes quasi-isomorphic to a bounded complex of finitely generated projectives. We use similar notation for $S$, where we will also wish to consider $\sfD^\mathrm{b}_\mathrm{tors}(\gr S)$ the full subcategory of complexes of finitely generated graded $S$-modules with finite length total cohomology. We shall use $\Sigma$ to denote the suspension in all of these categories.

\begin{ex}\label{ex_companion}
The companion category of $S$ is the $\base$-linear path category of the doubly infinite type $A$ quiver
\begin{displaymath}
\xymatrix{
\cdots \ar[r] & i-1 \ar[r] & i \ar[r] & i+1 \ar[r] & \cdots
}
\end{displaymath}
which provides the link with representation theory mentioned in the introduction. The companion category of $R$ is the Koszul dual of the above path category, i.e.\ it is given by the same diagram with square zero relations.
\end{ex}

\subsection{A very special case of the BGG correspondence}

In the course of performing the necessary computations to establish our geometric model in the next section it will be convenient to work with the polynomial ring $S$ rather than the exterior algebra $R$; this is partially for psychological reasons and partially because $S$ is hereditary. Koszul duality, in this rather special situation, tells us that we can work with whichever ring is most convenient in order to understand the bounded derived category. The salient points of the BGG correspondence \cite{BGG} are summarised in the following theorem which we will use freely in the sequel.

\begin{thm}\label{thm_bgg}
The full subcategory
\begin{displaymath}
\sfT = \{\Sigma^{i}\base(-i) \; \vert \; i\in \ZZ\}
\end{displaymath}
of $\sfD^\mathrm{b}(\gr R)$ is tilting. It is equivalent to the companion category of $S$ and so induces an equivalence of triangulated categories
\begin{displaymath}
\phi = \RHom(\sfT,-) \colon \sfD^\mathrm{b}(\gr R) \to \sfD^\mathrm{b}(\gr S)
\end{displaymath}
such that
\begin{itemize}
\item $\phi\circ (1) \cong \Sigma(-1) \circ \phi$;
\item $\phi$ sends perfect complexes to complexes with torsion cohomology.
\end{itemize}
In particular, $\phi$ restricts to an equivalence
\begin{displaymath}
\sfD^\mathrm{perf}(\gr R) \to \sfD^\mathrm{b}_\mathrm{tors}(\gr S).
\end{displaymath}
\end{thm}

As this fact is standard we don't include a proof, there are several modern references such as \cite{EFS}.

\subsection{Thick subcategories}

Let $\sfT$ be an essentially small triangulated category, for instance the bounded derived category of finitely generated graded modules over some graded ring. As earlier we shall denote the suspension in a triangulated category, e.g.\ $\sfT$, by $\Sigma$. We recall that a full replete subcategory $\sfM\subseteq \sfT$ is \emph{thick} if $\sfM$ is closed under the suspension, cones, and taking direct summands. In particular, $\sfM$ is a triangulated subcategory of $\sfT$. Given a set of objects $M\subseteq \sfT$ we denote by $\thick(M)$ the smallest thick subcategory containing $M$ and refer to it as the thick subcategory generated by $M$.

Since $\sfT$ is essentially small there is a set of thick subcategories and we denote it by $\Thick(\sfT)$. This collection is naturally a poset under inclusion and is in fact a complete lattice. The meets are provided by intersections and the join by
\begin{displaymath}
\sfM \vee \sfM' = \thick(\sfM \cup \sfM').
\end{displaymath}

%-------------------------------------------------------------------------------------------------------------------------------------------------

%-------------------------------------------------------------------------------------------------------------------------------------------------

\section{A geometric model for morphisms}\label{sec_model}

%-------------------------------------------------------------------------------------------------------------------------------------------------

\subsection{A combinatorial model}

Consider the poset $\ZZ\sqcup \{-\infty\}$ i.e.\ the integer line together with a point at $-\infty$ which, as the notation suggests, is minimal. We call an ordered pair $(a,b)$ an {\em arc} if $a, b \in \ZZ \cup \{-\infty\}$ with $a <b$ and we call $a$ and $b$ the {\em endpoints of the arc $(a,b)$}. We portray the arc $(a,b)$ by a curve connecting its two endpoints $a$ and $b$:

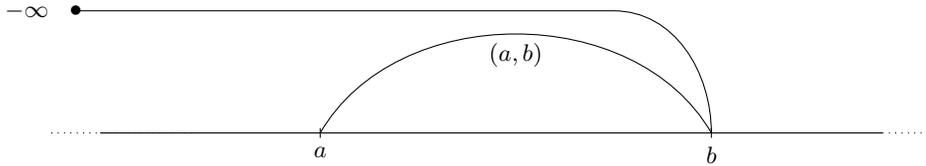
\begin{center}
\begin{tikzpicture}[scale =.65]
\tikzstyle{every node}=[font=\small]
\draw (-7.5,0) -- (8.5,0);
\draw[dotted] (-8.5,0) -- (-5.5,0);
\draw[dotted] (8.5,0) -- (9.5,0);
\draw (-3,0.1) -- (-3,-0.1) node[below]{$a$};
%\draw (1,0.1) -- (1,-0.1) node[below]{$m$};
\draw (5,0.1) -- (5,-0.1) node[below]{$b$};
%\draw (7,0.1) -- (7,-0.1) node[below]{$n$};
\node (-infty) at (-8,2.5) {$\bullet$};
\node (-infty) at (-9,2.5) {$-\infty$};

\path (-3,0) edge [out= 60, in= 120] node[below]{$(a,b)$} (5,0);
\draw (-8,2.5) -- (3,2.5) to[out=0,in = 90] (5,0);
%\path (5,0) arc (0:90:2.5cm) -- (-8.5,2.5);
%\path (1,0) edge [out= 60, in= 120] (7,0);

\end{tikzpicture}
\end{center}

Now let us fix an arc $(a,b)$. We call
\[
	l(a,b) = 	\begin{cases}
				(b - a) &\text{ if } a,b \in \ZZ \\
				\infty &\text{ if } a = -\infty
			\end{cases}
\]
the {\em length} of the arc $(a,b)$. We denote by 
\[
	\A = \{(a,b) \mid a,b \in \ZZ \cup \{-\infty\} \text{ with } a < b\}
\]
the set of arcs and by
\[
	\A_f = \{(a,b) \mid a,b \in \ZZ \text{ with } a < b\} = \{(a,b) \in \A \mid l(a,b)<\infty\}
\]
the set of arcs of finite length.

Recall that $S$ denotes the graded polynomial ring $S=k[y]$ with $y$ in degree $1$. There is a bijection between the set of arcs $\A$ and the set $\Ind(\gr S)$ of isomorphism classes of indecomposable graded $S$-modules given by
	\begin{eqnarray*}
		\varphi \colon \Ind(\gr S) & \to &  \A\\
		S\big/(y^i)(j) & \mapsto & (j-i,j)\\
		S(j) & \mapsto & (-\infty,j).
	\end{eqnarray*}
This restricts to a bijection between arcs of finite length and isomorphism classes of graded $S$-modules of finite length. In fact, as one easily checks, this bijection is compatible with the notion of length on either side
\begin{displaymath}
\length S/(y^i)(j) = l(j-i,i).
\end{displaymath}
	
	For any arc $(a,b) \in \A$ we denote by $M_{(a,b)}$ the graded $S$-module 
\begin{displaymath}
	\varphi^{-1}((a,b)) = 
	\begin{cases}
	S(b)  &\text{if} \; a = -\infty \\
	S/(y^{b-a})(b)  &\text{if} \; a\in \ZZ
	\end{cases}
\end{displaymath}
	We will frequently also consider $M_{(a,b)}$ as an object of $\sfD^\mathrm{b}(\gr S)$ by viewing it as a stalk complex concentrated in cohomological degree $0$. For the sake of convenience, we allow the notation $M_{(a,b)}$ with $a = b$ and identify
	$M_{(a,a)}$ with the zero object for any $a \in \ZZ \cup \{-\infty\}$.

This induces a natural bijection between sets of arcs, idempotent complete additive subcategories of $\gr S$, and idempotent complete and suspension-closed additive subcategories of $\sfD^\mathrm{b}(\gr S)$. Indeed, since $S$ is hereditary and $\gr S$ is a Krull-Schmidt category, either type of subcategory is just determined by the indecomposable modules it contains and these, as we have seen above, can be identified with arcs.

\subsection{Morphisms and extensions}

We now extend our bijection identifying arcs with indecomposable finitely generated graded $S$-modules to a description of morphisms and extensions in terms of the combinatorics of the arcs. This description will be particularly convenient in the sense that one can read off the corresponding cones in $\sfD^\mathrm{b}(\gr S)$.

\begin{defn}
	We say that two arcs $(a,b), (c,d) \in \A$ {\em touch} if $a \leq c \leq b \leq d$ or $c \leq a \leq d \leq b$. We say that they {\em cross} if $a < c < b < d$ or 
	$c < a < d < b$.
\end{defn}

There are five manners in which two arcs $(a,b)$ and $(c,d)$ can touch with $a \leq c \leq b \leq d$, depending on which of the inequalities are strict and which are equalities. They are illustrated, for finite length arcs, in Figure \ref{F:touching arcs}.

\begin{figure}
\centering
\begin{tikzpicture}[scale =.65]
\tikzstyle{every node}=[font=\small]
\node (x) at (-0.5,0.5) {(i)};
\draw (0,0) -- (5,0);
\draw (1,0.1) -- (1,-0.1) node[below]{$\mathstrut a=c$};
%\draw (-7,0.1) -- (-7,-0.1) node[below]{$c$};
\draw (4,0.1) -- (4,-0.1) node[below]{$\mathstrut b=d$};
%\draw (-5,0.1) -- (-5,-0.1) node[below]{$d$};

\path (1,0) edge [out= 60, in= 120] (4,0);
%\path (-7,0) edge [out= 60, in= 120] (-5,0);

\begin{scope}[xshift=200]
\node (x) at (-0.5,0.5) {(ii)};
\draw (0,0) -- (5,0);
\draw (0.5,0.1) -- (0.5,-0.1) node[below]{$\mathstrut a$};
\draw (1.5,0.1) -- (1.5,-0.1) node[below]{$\mathstrut c$};
\draw (3.5,0.1) -- (3.5,-0.1) node[below]{$\mathstrut b$};
\draw (4.5,0.1) -- (4.5,-0.1) node[below]{$\mathstrut d$};

\path (0.5,0) edge [out= 60, in= 120] (3.5,0);
\path (1.5,0) edge [out= 60, in= 120] (4.5,0);
\path[dashed] (0.5,0) edge [out= 60, in= 120] (1.5,0);
\path[dotted] (3.5,0) edge [out= 60, in= 120] (4.5,0);
\path[gray] (0.5,0) edge [out= 60, in= 120] (4.5,0);
\path[gray] (1.5,0) edge [out= 60, in= 120] (3.5,0);
\end{scope}

\begin{scope}[xshift=400]
\node (x) at (-0.5,0.5) {(iii)};
 \draw (0,0) -- (5,0);
\draw (0.5,0.1) -- (0.5,-0.1) node[below]{$\mathstrut a$};
\draw (2.5,0.1) -- (2.5,-0.1) node[below]{$\mathstrut c=b$};
\draw (4.5,0.1) -- (4.5,-0.1) node[below]{$\mathstrut d$};
%\draw (-5,0.1) -- (-5,-0.1) node[below]{$d$};

\path (0.5,0) edge [out= 60, in= 120] (2.5,0);
\path (2.5,0) edge [out= 60, in= 120] (4.5,0);
\path[gray] (0.5,0) edge [out= 60, in= 120] (4.5,0);
\end{scope}

\begin{scope}[xshift=100, yshift=-100]
\node (x) at (-0.5,0.5) {(iv)};
 \draw (0,0) -- (5,0);
\draw (0.5,0.1) -- (0.5,-0.1) node[below]{$\mathstrut a=c$};
%\draw (1.5,0.1) -- (1.5,-0.1) node[below]{$\mathstrut c$};
\draw (3.5,0.1) -- (3.5,-0.1) node[below]{$\mathstrut b$};
\draw (4.5,0.1) -- (4.5,-0.1) node[below]{$\mathstrut d$};

\path (0.5,0) edge [out= 60, in= 120] (3.5,0);
\path (0.5,0) edge [out= 60, in= 120] (4.5,0);
\path[dotted] (3.5,0) edge [out= 60, in= 120] (4.5,0);
\end{scope}

\begin{scope}[xshift=300, yshift=-100]
\node (x) at (-0.5,0.5) {(v)};
\draw (0,0) -- (5,0);
\draw (0.5,0.1) -- (0.5,-0.1) node[below]{$\mathstrut a$};
\draw (1.5,0.1) -- (1.5,-0.1) node[below]{$\mathstrut c$};
%\draw (3.5,0.1) -- (3.5,-0.1) node[below]{$\mathstrut d$};
\draw (4.5,0.1) -- (4.5,-0.1) node[below]{$\mathstrut b=d$};

\path (0.5,0) edge [out= 60, in= 120] (4.5,0);
\path (1.5,0) edge [out= 60, in= 120] (4.5,0);
\path[dashed] (0.5,0) edge [out= 60, in= 120] (1.5,0);
\end{scope}

\end{tikzpicture}
\caption{The arcs $(a,b)$ and $(c,d)$ touch \label{F:touching arcs}}
\end{figure}
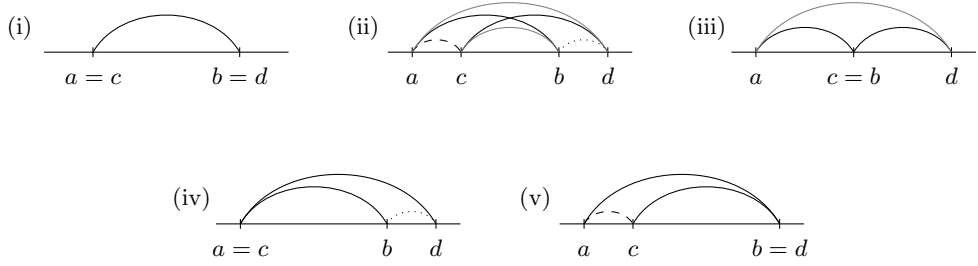

\begin{lem}\label{L:morphisms}
	For any two arcs $(a,b)$ and $(c,d)$ in $\A$ we have
		\[
			\Hom_{\gr S}(M_{(a,b)},M_{(c,d)}) \cong 	\begin{cases}
													\base \text{ if $(a,b)$ and $(c,d)$ touch with }  a \leq c < b \leq d \\
													0 \text{ otherwise.}
												\end{cases}
		\]
	Moreover, if $\varphi \colon M_{(a,b)} \to M_{(c,d)}$ is a non-trivial morphism then we have
		\begin{eqnarray*}
			\ker(\varphi) \cong M_{(a,c)} \\
			\coker(\varphi) \cong M_{(b,d)}.
		\end{eqnarray*}

\end{lem}

Before giving the proof let us illustrate the statement using Figure \ref{F:touching arcs}.

\begin{ex}
There are non-trivial morphisms from $M_{(a,b)}$ to $M_{(c,d)}$ for the arcs $(a,b)$ and $(c,d)$ in pictures (i), (ii), (iv) and (v) of Figure \ref{F:touching arcs}. The kernels and cokernels are marked in the same figure: picking a non-trivial morphisms in each of these cases, the arc corresponding to its kernel is drawn with a dashed line, and the arc corresponding to its cokernel with a dotted line. Note that  in (i) the morphism is necessarily an isomorphism, while in (iv) it is a monomorphism and in (v) an epimorphism.
\end{ex}

\begin{proof}
	Assume first that $a \in \ZZ$ and set 
\begin{displaymath}
	j = b, \;\; i =b-a, \;\; l=d, \;\; \text{and} \;\; k = d-c. 
\end{displaymath}
	Then $M_{(a,b)}$ is the module $S/(y^i)(j)$ and $M_{(c,d)}$ is $S/(y^k)(l)$. If $j > l$ or $j \leq l-k$ then $S/(y^k)(l)$ has no non-trivial element in degree $-j$, where the generator
	of $S / (y^i)(j)$ lives, thus there are no non-zero morphisms from $M_{(a,b)}$ to $M_{(c,d)}$. On the other hand, if 
	\begin{displaymath}
	j \leq l \;\; \text{and} \;\; l-k<j
	\end{displaymath}
	consider multiplication by $\lambda y^{(l-j)}$ for a $\lambda \in \base$, which
	sends the generator $1(j) \in S / (y^i)(j)$ to $\lambda y^{(l-j)}(l) \in S / (y^k)(l)$. This map is well-defined if and only if $y^i(j)$, which is trivial in $S / (y^i)(j)$, gets mapped to zero,
	which, for $\lambda \neq 0$, is the case if and only if $l-j+i \geq k$. The modules in question are cyclic so these are the only possible degree zero maps from $S / (y^i)(j)$ to $S / (y^k)(l)$. Thus we conclude that
	$\Hom_{\gr S}(M_{(a,b)},M_{(c,d)}) \cong \base$ if and only if 
	\begin{displaymath}
	j-i \leq l-k < j \leq l, \;\; \text{i.e.\ if and only if} \;\; a \leq c < b \leq d
	\end{displaymath}
	and it is zero otherwise. In the former case, the kernel of multiplication by $\lambda y^{(l-j)}$, for $\lambda \neq 0$, is $S / (y^{(l-k)-(j-i)})(l-k)$ and its cokernel is $S / (y^{l-j})(l)$, that is $M_{(a,c)}$ and $M_{(b,d)}$ respectively.

	Consider now the case where $a = -\infty$ and $c>-\infty$ and set
	\[
	 j=b, \;\; l = d, \; \; \text{and} \; \;k = d-c.
	\]
	Then $M_{(a,b)}$ is the module $S (j)$ with generator living in degree $-j$ and $M_{(c,d)}$ is $S/(y^k)(l)$. There are no non-zero morphisms from $S (j)$ to $S/(y^k)(l)$ if the latter has no non-trivial element in degree $-j$, which is the case if and only if $j>l$ or $j \leq l-k$.  Multiplication by $\lambda y^{(l-j)}$ for any $\lambda \in \base$ gives rise to a well-defined morphism from $S (j)$ to any $S/(y^k)(l)$ with a non-trivial element in degree $-j$. These are the only possible degree zero maps from $S(j)$ to $S/(y^k)(l)$ by virtue of the modules being cyclic, so we see
	$\Hom_{\gr S}(M_{(a,b)},M_{(c,d)}) \cong \base$ if and only if $l-k < j \leq l$ which is the case if and only if $a \leq c < b \leq d$ and it is zero otherwise. In the former case, the kernel of multiplication by $\lambda y^{(l-j)}$, for $\lambda \neq 0$, is $S(l-k)$ and its cokernel is $S/(y^{(l-j)})(l)$, that is $M_{(a,c)}$ and $M_{(b,d)}$ respectively.
	
	Finally consider the case where $a = c = -\infty$ and set
	\[
	 j=b, \;\; \text{and} \; \;l=d.
	\]
	It follows by the same argument given above that $\Hom_{\gr S}(M_{(a,b)},M_{(c,d)}) \cong \base$ if and only if $j \leq l$ which is the case if and only if  $a \leq c < b \leq d$ and is zero otherwise. In the former case, multiplication by $\lambda y^{(l-j)}$, for $\lambda\neq 0$ is a monomorphism with cokernel $S/(y^{(l-j)})(l)$, which corresponds to $M_{(b,d)}$.
\end{proof}

\begin{lem}\label{L:Exts}
	For any two arcs $(a,b)$ and $(c,d)$ in $\A$ we have
		\[
			\Hom_{\sfD^\mathrm{b}(\gr S)}(M_{(c,d)}, \Sigma M_{(a,b)})  \cong	\begin{cases}
													\base \text{ if $(a,b)$ and $(c,d)$ touch with } a < c \leq b < d \\
													0 \text{ otherwise.}
												\end{cases}
		\]
	Moreover, if $\Hom_{\sfD^\mathrm{b}(\gr S)}(M_{(c,d)}, \Sigma M_{(a,b)}) \cong \base$ we have, for each non-zero element, a non-split distinguished triangle
		\begin{eqnarray*}
			M_{(a,b)} \to B \to M_{(c,d)} \to \Sigma M_{(a,b)}
		\end{eqnarray*}
	in $\sfD^\mathrm{b}(\gr S)$ with $B \cong M_{(c,b)} \oplus M_{(a,d)}$, i.e.\ $B$ is the fibre of any non-zero map $M_{(c,d)}\to \Sigma M_{(a,b)}$.
\end{lem}

Again let us illustrate the lemma using the figure before proving it.

\begin{ex}
There is a non-trivial $\Ext^1$ from $M_{(c,d)}$ to $M_{(a,b)}$ for the arcs $(a,b)$ and $(c,d)$ in pictures (ii) and (iii) of Figure \ref{F:touching arcs}. The arcs corresponding to the indecomposable summands of the middle term of the corresponding non-split distinguished triangle are marked in grey.
\end{ex}

\begin{proof}
	If $c = -\infty$, then $M_{(c,d)}$ is the projective module $S(d)$ and so of course we have $\Ext^1(M_{(c,d)}, M_{(a,b)}) = 0$. Assume then that $c \in \ZZ$ and set 
\begin{displaymath}
	j = d \;\; \text{and} \;\; i =d-c.
\end{displaymath}
Then $M_{(c,d)}$ is $S / (y^i)(j)$ which has minimal projective resolution
	\[	
		0 \to S(j-i) \to S(j) \to 0
	\]
	In our notation reflecting the model this can be written
	\[
		0 \to M_{(-\infty,c)} \to M_{(-\infty,d)} \to 0
	\]
	Applying $\Hom_{\gr S}(-,M_{(a,b)})$ yields
	\[
		\xymatrix{0\ar[r] &\Hom_{\sfD^\mathrm{b}(\gr S)}(M_{(-\infty,d)},M_{(a,b)}) \ar[r]^-f& \Hom_{\sfD^\mathrm{b}(\gr S)}(M_{(-\infty,c)},M_{(a,b)}) \ar[r]^-g & 0}.
	\]
	By Lemma \ref{L:morphisms} we have $\ker(g) \cong \base$ if and only if  $a < c \leq b$ and zero otherwise. If $a < c \leq b$ then, since $c < d$, we either have $a < c < d \leq b$ or $a < c \leq b < d$ and we have $\im(f) \cong \base$ if and only if $a < c < d \leq b$ and it is zero if and only if $a < c \leq b<d$. So we get that
	\[
		\Ext^1_{\sfD^\mathrm{b}(\gr S)}(M_{(c,d)},M_{(a,b)})  \cong \ker(g) \big/ \im(f) \cong \base	
	\]
	if and only if $(a,b)$ and $(c,d)$ touch with $a < c \leq b < d$ and it is $0$ otherwise.

	For the final statement, assume that our arcs touch in the appropriate manner so that $\Ext^1_{\sfD^\mathrm{b}(\gr S)}(M_{(c,d)},M_{(a,b)})  \cong \base$. One can easily verify that there is a non-split short exact sequence
	\begin{displaymath}
	\xymatrix{
	0 \ar[r] & M_{(a,b)} \ar[rr]^-{\begin{pmatrix} 1 & x^{d-b} \end{pmatrix}} &&	M_{(c,b)}\oplus M_{(a,d)} \ar[rr]^-{\begin{pmatrix} x^{d-b} \\ -\pi \end{pmatrix}} && 	M_{(c,d)} \ar[r] & 0
	}
	\end{displaymath}
(keeping in mind that if $c=b$ then $M_{(c,b)} = 0$) which corresponds to a non-trivial map
	\begin{displaymath}
	M_{(c,d)}\stackrel{\eta}{\to} \Sigma M_{(a,b)}.
	\end{displaymath}
	The assertion now follows as we have shown the Ext is $1$-dimensional so the other non-trivial classes are obtained by scaling this map and hence have the same fibre, namely the middle term of the above sequence.
\end{proof}

%-------------------------------------------------------------------------------------------------------------------------------------------------

%-------------------------------------------------------------------------------------------------------------------------------------------------

\section{Thick subcategories}\label{sec_thick}

In this section we use our geometric model to describe the thick subcategories of $\sfD^\mathrm{b}(\gr S)$. Our first description is in terms of collections of arcs satisfying a certain closure condition, which we call saturation, reflecting closure under cones and is very well adapted to the model. It seems fitting, following \cite{IT} and \cite{IS}, to describe these sorts of classifications in terms of non-crossing partitions and we go on to provide such a description in Section~\ref{sec_ncp}.

%-------------------------------------------------------------------------------------------------------------------------------------------------

\subsection{Thick subcategories of the bounded derived category}

%-------------------------------------------------------------------------------------------------------------------------------------------------

\begin{defn}
	We call a set of arcs $\mathcal{S} \subseteq \A$ {\em saturated} if whenever two arcs $(a,b)$ and $(c,d)$ in $\mathcal{S}$ touch with $a \leq c \leq b \leq d$, then those of $(a,c), (c,b), (b,d)$ and $(a,d)$ that are arcs also lie in $\mathcal{S}$. We denote the collection of saturated sets of arcs in $\A$ by $\Sat(\A)$.
	
	If $\mathcal{U} \subseteq \A$ is a set of arcs, we call the smallest saturated subset of $\A$ containing $\mathcal{U}$ the {\em saturation of $\mathcal{U}$} and denote it by $\sat(\mathcal{U})$.
\end{defn}

\begin{rem}
	Note that if $\mathcal{U} \subseteq \A_f$ is a set of arcs of finite length, then so is its saturation: $\sat(\mathcal{U}) \subseteq \A_f$.
\end{rem}

Evidently there are lattice structures on the collection of subsets of arcs $\sfP(\A)$, i.e.\ the powerset of $\A$, and on the collection of saturated sets of arcs $\Sat(\A)$. In both cases the poset structure is given by inclusion and the meet is given by intersection\textemdash{}one sees straight away that if $\mcS$ and $\mcS'$ are saturated then so is $\mcS\cap \mcS'$. Thus both of these posets are complete lattices. We note that joins in $\Sat(\A)$ do not necessarily coincide with those in $\sfP(\A)$ as one may have to take the saturation of the union to obtain the join in $\Sat(\A)$. By the preceding remark $\Sat(\A_f)$, the collection of saturated sets of finite arcs, is almost a sublattice of $\Sat(\A)$. More precisely, it is closed under arbitrary meets and joins in $\Sat(\A)$, however, it doesn't contain the maximal element $\A$.

Our first aim in this section is to identify $\Sat(\A)$ with the lattice of thick subcategories of $\sfD^\mathrm{b}(\gr S)$. To this end we wish to define a pair of lattice maps

\begin{displaymath}
\xymatrix{
\Sat(\A) \ar[rr]<0.5ex>^-\tau \ar@{<-}[rr]<-0.5ex>_-\sigma && \Thick(\sfD^\mathrm{b}(\gr S))
}
\end{displaymath}

comparing these two structures as follows: given a saturated set of arcs $\mcS$ and a thick subcategory $\sfM$ we claim
\begin{displaymath}
\tau(\mcS) = \add\{\Sigma^lM_{(a,b)} \; \vert \; l\in \ZZ \; \text{and} \; (a,b)\in \mcS\}
\end{displaymath}
and
\begin{displaymath}
\sigma(\sfM) = \{(a,b)\in \A \; \vert \; M_{(a,b)}\in \sfM\}.
\end{displaymath}
are well defined morphisms of lattices. It is clear that both assignments, if well defined, are lattice maps preserving arbitrary meets and joins; the rest of this section is essentially devoted to verifying that $\tau(\mcS)$ is indeed thick and $\sigma(\sfM)$ is indeed saturated.

\begin{lem}\label{lem_sigma}
Let $\sfM$ be a thick subcategory of $\sfD^\mathrm{b}(\gr S)$. Then $\sigma(\sfM)$ is a saturated set of arcs.
\end{lem}
\begin{proof}
Suppose two arcs $(a,b), (c,d) \in \sigma(\sfM)$ touch, where without loss of generality we may assume $a \leq c \leq b \leq d$. 
	
	We first show that those of $(a,c)$ and $(b,d)$ that are arcs lie in $\sigma(\sfM)$. If we have $a \leq c = b \leq d$ then $(a,c) = (a,b)$ and $(b,d) = (c,d)$ which trivially lie in $\sigma(\sfM)$. If, on the other hand, we have $a \leq c < b \leq d$ then by Lemma \ref{L:morphisms} 
	we know that $\Hom_{\sfD^\mathrm{b}(\gr S)}(M_{(a,b)},M_{(c,d)}) \cong \base$. We choose a non-trivial morphism $\varphi \colon M_{(a,b)} \to M_{(c,d)}$. Since $\sfM$ is thick, the cone of $\varphi$ (which is described in Lemma~\ref{L:morphisms})
	\[
	\cone(\varphi) \cong \Sigma \ker \varphi \oplus \coker \varphi \cong \Sigma M_{(a,c)} \oplus M_{(b,d)}
	\]
	lies in $\sfM$, and therefore those of $(a,c)$ and $(b,d)$ that are arcs must also in $\sigma(\sfM)$. 
	
	We next show that those of $(c,b)$ and $(a,d)$ which are arcs also lie in $\sigma(\sfM)$. If we have $a = c$, then $(c,b) = (a,b)$ and $(a,d) = (c,d)$ by assumption lie in $\sigma(\sfM)$. Similarly, if we have $b = d$, then $(c,b) = (c,d)$ and $(a,d) = (a,b)$ lie in $\sigma(\sfM)$. In the final remaining case we have $a < c \leq b < d$, and by Lemma \ref{L:Exts} there is a non-split distinguished triangle
	\begin{eqnarray*}
			M_{(c,d)} \to B \to M_{(a,b)} \to \Sigma M_{(c,d)}
	\end{eqnarray*}
	with $B \cong M_{(c,b)} \oplus M_{(a,d)}$. Since $\sfM$ is closed under extensions and summands we deduce that those of $(c,b)$ and $(a,d)$ which are arcs also lie in $\sigma(\sfM)$. Thus $\sigma(\sfM)$ is a saturated set of arcs as claimed.
\end{proof}

Showing that $\tau$ takes values in thick subcategories is a little less straightforward and requires some preparation. The crux of the matter is that, since our model only describes indecomposable objects and the cones on morphisms between them, we need to know that understanding such cones completely determines the thick subcategory a collection of indecomposable modules generates. As any object is a direct sum of finitely many indecomposable objects it is enough to understand the situation for finite saturated sets of arcs. We will, in fact, prove a stronger result and give a description of the thick subcategory generated by any finite set of arcs from which the result we need follows.

Let us fix a finite set $\mcF\subseteq \A$ of arcs and set $\mcS = \sat(\mcF)$ the saturation of $\mcF$. Clearly $\mcS$ is still a finite collection of arcs. Note that $\A$, together with the lexicographic order, is a linear order. In particular, any finite set of arcs has a minimal element and we denote by $(l,m)$ the minimal arc in $\mcS$. 

\begin{lem}\label{lem_minext}
For any $(a,b)\in \mcS$ distinct from $(l,m)$
\begin{displaymath}
\RHom(M_{(a,b)}, M_{(l,m))} \cong \Ext^1(M_{(a,b)}, M_{(l,m)}).
\end{displaymath}
\end{lem}
\begin{proof}
By Lemma~\ref{L:morphisms} in order for $\Hom(M_{(a,b)}, M_{(l,m))}\neq 0$ we would need $a\leq l < b \leq m$, but this would violate the minimality of $(l,m)$.
\end{proof}

We wish to use $(l,m)$ to decompose $\mcS$. This will come down to the usual yoga of perpedicular subcategories (albeit viewed through the lens of our model); keeping this in mind we will frequently tacitly identify $\mcS$ with $\tau(\mcS)$. The first observation is that the perpendicular of $(l,m)$ is again saturated.

\begin{lem}\label{lem_satpres}
The subset
\begin{displaymath}
\mcS\perp(l,m) = \{(a,b)\in \mcS\;\vert\; \text{neither}\;\; l<a\leq m< b \;\; \text{nor} \;\; a \leq l < b \leq m \;\; \text{hold}\}
\end{displaymath}
is again saturated.
\end{lem}

\begin{rem}
We should first comment that by Lemmas~\ref{L:morphisms} and \ref{L:Exts} the above is really the left perpendicular to $(l,m)$ in the sense one would guess. Secondly, as a consequence of Lemma~\ref{lem_minext}, for arcs in $\mcS$ distinct from $(l,m)$ it is enough for the first condition to fail, the latter failing automatically.
\end{rem}

\begin{proof}
Suppose $(a,b)$ and $(c,d)$ lie in $\mcS\perp(l,m)$ and touch with $a\leq c\leq b\leq d$. We know, since $\mcS$ is saturated, that any of $(a,c), (c,b), (a,d)$, and $(b,d)$ which are arcs lie in $\mcS$. We need to verify that they also lie in $\mcS\perp(l,m)$. 
\begin{itemize}
\item If $(a,c)$ were an arc with $(a,c) \notin \mcS\perp(l,m)$, then we would have $l < a \leq m < c \leq b$ contradicting $(a,b) \in \mcS\perp(l,m)$.
\item If $(c,b)$ were an arc with $(c,b) \notin \mcS\perp(l,m)$, then we would have $l < c \leq m < b \leq d$ contradicting $(c,d) \in \mcS\perp(l,m)$.
\item If $(b,d)$ were an arc with $(b,d) \notin \mcS\perp(l,m)$, then we would have $l < b \leq m < d$. By minimality of $(l,m) \in \mcS$ we have $l < a \leq c \leq m < d$. Otherwise, if $a \leq l < b \leq m$ this would imply $(a,b) = (l,m)$ contradicting $(a,b) \in \mcS\perp(l,m)$. However, this contradicts $(c,d) \in \mcS\perp(l,m)$.
\item If $(a,d)$ were an arc with $(a,d) \notin \mcS\perp(l,m)$, then we would have $l < a \leq m < d$. If $l < a \leq m < b < d$ this contradicts $(a,b) \in \mcS\perp(l,m)$ and otherwise, if $l < c \leq b \leq m < d$ this contradicts $(c,d) \in \mcS\perp(l,m)$.
\end{itemize}
\end{proof}

Before coming to the decomposition statement we need to introduce some standard terminology.

\begin{defn}
Let $\sfT$ be a $k$-linear triangulated category. An object $E\in \sfT$ is \emph{exceptional} if 
\begin{displaymath}
\RHom(E,E) \cong k,
\end{displaymath}
i.e.\ the derived endomorphism ring $E$ is simply $k$ concentrated in degree zero. 

An \emph{exceptional collection} $(E_1,\ldots, E_n)$ in $\sfT$ is a ordered sequence of exceptional objects $E_i$ such that $\RHom(E_j,E_i)=0$ for $j>i$. It is \emph{strong} if $\RHom(E_i,E_j) \cong \Hom(E_i,E_j)$ for $i\leq j$. 
\end{defn}

We have, without making it explicit, already encountered many exceptional objects.

\begin{ex}
By Lemmas~\ref{L:morphisms} and \ref{L:Exts} every indecomposable object of $\sfD^\mathrm{b}(\gr S)$ is exceptional.
\end{ex}

The following theorem gives many examples in $\sfD^\mathrm{b}(\gr S)$ of exceptional collections; besides being rather natural in its own right it will serve a technical role in the paper by allowing us to reduce to understanding indecomposable objects when considering thick subcategories.

\begin{thm}\label{thm_exceptional}
There is a strong exceptional collection
\begin{displaymath}
\mathbf{E} = (M_{(l_1,m_1)},\ldots, M_{(l_n,m_n)}) \;\; \text{with} \;\; (l_i,m_i)\in \mcS \;\text{for}\; 1\leq i\leq n
\end{displaymath}
in $\sfD^\mathrm{b}(\gr S)$ such that $\sat((l_i,m_i)\;\vert\; 1\leq i \leq n) = \mcS$. Moreover, there is an isomorphism
\begin{displaymath}
\End(\bigoplus_{i=1}^n M_{(l_i,m_i)}) \cong \base A_{r_1} \times \cdots \times \base A_{r_k}
\end{displaymath}
where $r_1+\cdots +r_k = n$ and each quiver is linearly oriented.
\end{thm}
\begin{proof}
We take $(l_1,m_1) = (l,m)$, the minimal arc in $\mcS$. We set $\mcS_0 = \mcS$ and iteratively define for $i \geq 1$
\[
 \mcS_i = \mcS_{i-1}\perp(l_i,m_i),
\]
where $(l_i,m_i)$ is the minimal arc in $\mcS_{i-1}$. This is well-defined: by iteratively applying Lemma~\ref{lem_satpres} we know $\mcS_{i-1}\perp(l_i,m_i)$ is saturated and, since it is clearly finite, there exists such a minimal arc $(l_i,m_i)$ in $\mcS_{i-1}\perp(l_i,m_i)$. As the cardinality of the $\mcS_i$ decreases this process must terminate in finitely many, say $n$, steps and we obtain a sequence of arcs $(l_i,m_i)$ for $1\leq i \leq n$. Each $M_{(l_i,m_i)}$ is exceptional by Lemma~\ref{L:Exts} and
\begin{displaymath}
\RHom(M_{(l_j,m_j)}, M_{(l_i,m_i)}) = 0 \;\; \text{for} \;\; j>i
\end{displaymath}
by construction, so the sequence is an exceptional collection as claimed.

It is clear that $\sat((l_i,m_i)\;\vert\; 1\leq i \leq n) \subseteq \mcS$. We prove the reverse containment by induction on $n$. For the base case, suppose $\mcS\perp(l_1,m_1) = \varnothing$ but there is some $(a,b)\neq (l_1,m_1)$ in $\mcS$. Then by Lemma~\ref{lem_minext} there must be an $\Ext^1$ from $(a,b)$ to $(l_1,m_1)$ and so using Lemma~\ref{L:Exts} we would have to have $l_1<a\leq m_1<b$. By the same lemma we know the fibre of any such non-trivial map would have a summand indexed by $(l_1,b)$. But $(l_1,b)\in \mcS\perp(l_1,m_1)$ which is absurd.

Suppose then that we know the statement for maximal such sequences of length less than $n$ and let $(a,b)\in \mcS$. Denoting $\sat((l_i,m_i)\;\vert\; 1\leq i \leq n)$ by $\mcS'$ we need to show that $(a,b)\in \mcS'$.
By the inductive hypothesis, we have
\[
 \mcS\perp(l_1,m_1) = \sat((l_i,m_i)\; \vert \; 2\leq i \leq n) \subseteq \mcS'.
\]
Therefore, if we had $(a,b) \in \mcS\perp(l_1,m_1)$ we would be done, so assume this is not the case.
Then there is a non-trivial map from $M_{(a,b)}$ to some shift of $M_{(l_1,m_1)}$ and we know it must be an $\Ext^1$ by Lemma~\ref{lem_minext}. Thus $l_1< a\leq m_1 <b$ and the fibre of such a non-trivial map has summands indexed by $(a,m_1)$ and $(l_1,b)$ which lie in $\mcS\perp(l_1,m_1) \subseteq \mcS'$. There is then a degree $0$ map $M_{(l_1,m_1)} \to M_{(l_1,b)}$, by Lemma \ref{L:morphisms}, from whose cone we obtain $(m_1, b)$ which is also in $\mcS'$ since $(l_1,m_1)$ and $(l_1,b)$ are. Thus $(a,m_1)$ and $(m_1,b)$ lie in $\mcS'$ and glue to give $(a,b)\in \mcS'$. This proves $\mcS'=\mcS$ as claimed.

We next show $\mathbf{E}$ is strong. Say $i<j$ and consider $\Ext^1(M_{(l_i,m_i)},M_{(l_j,m_j)})$. This is non-trivial if and only if $l_j<l_i\leq m_j <m_i$. But, by construction, $(l_i,m_i)$ is minimal under the lexicographic ordering in $\mcS\perp(l_{i-1},m_{i-1})$ i.e.\ we know that $(l_i,m_i) <_\mathrm{lex} (l_j,m_j)$ and so $l_j\nless l_i$. Hence the only possible morphisms are of degree zero as required.

To compute the endomorphism algebra, we observe that if there is a map $M_{(l_i,m_i)} \to M_{(l_j,m_j)}$ for $i < j$, then it factors via $M_{(l_{i+1},m_{i+1})}$: indeed, if we are given such a map then by Lemma \ref{L:morphisms} we have
\[
 l_i \leq l_j < m_i \leq m_j.
\]
In fact, we must have 
\[
 l_i = l_j < m_i < m_j.
\]
To see this, first observe that if $l_i < l_j < m_i \leq m_j$ then, if the last inequality were strict this would contradict $(l_j, m_j) \in \mcS_{i-1}\perp(l_i,m_i)$. Thus we must have $l_i = l_j$ or $m_i = m_j$. If $m_i = m_j$, then because $\mcS_{i-1}$ is saturated and both $(l_i,m_i)$ and $(l_j,m_j)$ lie in $\mcS_{i-1}$, we get that $(l_i,l_j) \in \mcS_{i-1}$, contradicting the minimality of $(l_i,m_i)$ in $\mcS_{i-1}$. This shows the only possible configuration is $l_i = l_j < m_i < m_j$. By minimality of $(l_{i+1},m_{i+1}) \in \mcS_i$ we deduce that
\begin{align}\label{E:factoring}
 l_i = l_{i+1}=l_j < m_i < m_{i+1} < m_j,
\end{align}
and the factorisation assertion follows from Lemma \ref{L:morphisms}. It follows directly from (\ref{E:factoring}) and Lemma \ref{L:morphisms} that the endomorphism algebra is given by a product of type $A$ quivers with no relations.
\end{proof}

\begin{cor}\label{lem_tau}
Let $\mcS$ be a saturated set of arcs as above. Then $\tau(\mcS)$ is a thick subcategory of $\sfD^\mathrm{b}(\gr S)$.
\end{cor}
\begin{proof}
By definition $\tau(\mcS)$ is closed under sums, summands, and suspensions. Thus it is enough to show it is also closed under taking cones. Computing the cone between two objects involves only finitely many indecomposables and so it is sufficient to demonstrate that $\tau(\mcS)$ is thick when $\mcS$ is finite. 

By Lemmas~\ref{L:morphisms} and \ref{L:Exts} the fact that $\mcS$ is saturated implies that $\tau(\mcS)$ is closed under taking cones of morphisms between indecomposables. This is all one needs to check: by the previous theorem, and the remark following it, $\thick(\tau(\mcS))$ is equivalent to the bounded derived category of $\base A_n$ for some $n$ and any subcategory of such a category which is closed under taking cones of maps between indecomposable objects is thick.
\end{proof}

With these results in hand we can now dispose easily of the classification result.

\begin{thm}\label{thm_saturatedclassification}
	The morphisms $\tau$ and $\sigma$ give an explicit lattice isomorphism
	\begin{displaymath}
	\Sat(\A) \cong \Thick(\sfD^\mathrm{b}(\gr S)),
	\end{displaymath}
	which restricts to an isomorphism
	\begin{displaymath}
	\Sat(\A_f) \cong \Thick(\sfD^\mathrm{b}_\mathrm{tors}(\gr S)).
	\end{displaymath}
\end{thm}

\begin{proof}
We have proved that $\tau$ and $\sigma$ are well defined maps of lattices so it just remains to show that they are indeed inverse. This boils down to chasing through the definitions as follows.

	Let $\mathcal{S}$ be a saturated set of arcs. We have
	\[
		\sigma\tau(\mcS) = \sigma(\add \{\Sigma^l M_{(a,b)} \mid l \in \ZZ \; \text{and}\; (a,b) \in \mathcal{S} \}) = \{(a,b)\in \A \; \vert\; (a,b)\in \mcS\} =  \mathcal{S},
	\]
	so $\sigma$ is a left inverse for $\tau$.
	On the other hand, let $\sfM$ be a thick subcategory of $\sfD^\mathrm{b}(\gr S)$. Then
	\begin{align*}
		\tau\sigma(\sfM) &= \add \{ \Sigma^l M_{(a,b)} \mid l \in \ZZ \; \text{and} \; (a,b) \in \sigma(\sfM)\}) \\
		&= \add \{ \Sigma^l M_{(a,b)} \mid l \in \ZZ\; \text{and}\; M_{(a,b)} \in \sfM \}) \\
		&= \sfM,
	\end{align*}
	since $\sfM$ is thick. Thus $\tau$ and $\sigma$ are inverse.
	
	The restricted bijection follows immediately by recalling that for any finite length arc $(a,b)$ the corresponding module $M_{(a,b)}$ has finite length and so $\tau$ takes $\A_f$ to $\sfD^\mathrm{b}_\mathrm{tors}(\gr S)$.
\end{proof}

\begin{cor}
	There is a commutative diagram
	\begin{displaymath}
	\xymatrix{
	\Sat(\A) \ar[r]^-\sim & \Thick(\sfD^\mathrm{b}(\gr R)) \\
	\Sat(\A_f) \ar[u] \ar[r]_-\sim & \Thick(\sfD^\mathrm{perf}(\gr R)) \ar[u]
	}
	\end{displaymath}
	where the horizontal maps are lattice isomorphisms and the vertical maps are lattice inclusions.
\end{cor}

\begin{proof}
This follows immediately from the theorem above by applying the BGG correspondence (as recalled in Theorem~\ref{thm_bgg}).
\end{proof}

%--------------------------------------------------------------------------------------------------------------------

\subsection{Thick subcategories and non-crossing partitions}\label{sec_ncp}
We now compare $\Sat(\A)$ to the lattice of non-crossing partitions of $\ZZ\sqcup\{-\infty\}$. Our main result is that they are isomorphic lattices in such a way that $\Sat(\A_f)$ is identified with non-crossing partitions of $\ZZ$. As a consequence we see that thick subcategories of $\sfD^\mathrm{b}(\gr R)$ are in bijection with non-crossing partitions of $\mathbb{Z}\sqcup \{-\infty\}$.

\begin{defn}
	Let $(\mathcal{Z}, < ) $ be a partially ordered set. A non-crossing partition of $\mathcal{Z}$ is a partition
		\[
			\mathcal{Z} = \bigsqcup_{i \in I} B_i
		\]
	such that if we have $a,b \in B_i$ and $c,d \in B_j$ for $i \neq j \in I$, then neither $a < c < b < d$ nor $c < a < d < b$ occurs in $\mathcal{Z}$.
	
	We denote the collection of non-crossing partitions of $\mathcal{Z}$ by $\NC(\mathcal{Z})$.
\end{defn}

The set of non-crossing partitions $\NC(\mathcal{Z})$ forms a poset with ordering given by reverse refinement. Thus the bottom and top elements are
\begin{displaymath}
\mathcal{Z} = \bigsqcup_{z\in \mathcal{Z}} \{z\} \;\; \text{and} \;\; \mathcal{Z} = \mathcal{Z}
\end{displaymath}
respectively, i.e.\ the finest and coarsest partitions. One can take the meet of two non-crossing partitions $B = \{B_i\;\vert\; i\in I\}$ and $C = \{C_j\;\vert\;j\in J\}$ in the usual way
\begin{displaymath}
B\wedge C = \{B_i\cap C_j \; \vert \; (i,j) \in I\times J\},
\end{displaymath}
it being evident that $B\wedge C$ is still non-crossing. Thus $\NC(\mathcal{Z})$ is a complete lattice. We note that it is not a sublattice of the lattice of partitions of $\mathcal{Z}$ as, in general, one has to take the least non-crossing refined by the usual join of partitions. We can regard $\NC(\ZZ)$ as  embedded in $\NC(\ZZ \cup \{-\infty\})$ as the non-crossing partitions where the singleton $\{-\infty\}$ is a block. In this way $\NC(\ZZ)$ would be a sublattice except that it fails to contain the maximal element. 

We now give assignments relating $\NC(\ZZ\sqcup\{-\infty\})$ and $\Sat(\A)$ which, as in the last section, we will show are well defined after describing them. For the sake of brevity we will set $\mathcal{Z} = \ZZ\sqcup\{-\infty\}$ in what follows.

Given a non-crossing partition $B=\{B_i\;\vert\;i\in I\}$ we define a collection of arcs
\begin{displaymath}
\alpha(B) = \bigcup_{i\in I}\{(a,b)\in \A \; \vert \; a,b,\in B_i\}.
\end{displaymath}
Given a saturated set of arcs $\mcS$ we first define, for each $a\in \mathcal{Z}$, a subset of $\mathcal{Z}$
\begin{displaymath}
\phi(\mcS)(a) = \{a\} \sqcup \{b\in \mathcal{Z} \; \vert \; (a,b)\in \mcS \; \text{or} \; (b,a)\in \mcS\}.
\end{displaymath}
We then associate to $\mcS$ the element of the powerset of $\mathcal{Z}$
\begin{displaymath}
\phi(\mcS) = \{\phi(\mcS)(a)\; \vert\; a\in \mathcal{Z}\},
\end{displaymath}
where we, as usual, identify redundant copies on the right-hand side. One could, of course, instead pick a set of representatives for the subsets on the right.

\begin{lem}\label{lem_alpha}
The assignment $\alpha$ defines a map of lattices
\begin{displaymath}
\alpha\colon \NC(\mathcal{Z}) \to \Sat(\A)
\end{displaymath}
which sends $\NC(\ZZ)$ to $\Sat(\A_f)$.
\end{lem}
\begin{proof}
	Let
	\[
			B = \{B_i\;\vert\; i\in I\}
	\]
	be a non-crossing partition of $\mathcal{Z}$ and
	\[
		\alpha(B) = \bigcup_{i \in I}\{(a,b) \in \A \mid a,b \in B_i\}
	\]
	as above. We first show that the set of arcs $\alpha(B)$ is saturated. Suppose then that $(a,b)$ and  $(c,d)$ lie in $\alpha(B)$ with $a \leq c \leq b \leq d$. By construction we have $a,b \in B_i$ and $c,d \in B_j$ for some $i,j \in I$. The collection $\alpha(B)$ will certainly contain whichever of $(a,c), (a,d), (b,c)$, and $(b,d)$ are arcs if $i = j$.
	
	If we have $i\neq j$, then since $B$ is a non-crossing partition, we cannot have $a < c < b < d$. Thus one of 
	$a = c$, $c = b$ or $b = d$ must hold. It would follow that $B_i \cap B_j \neq \varnothing$ contradicting the fact that $B$ is a partition. Therefore we must have $i = j$ and so $\alpha(B)$ contains the required arcs as indicated above.
	
Clearly $\alpha$ is a poset map and one sees easily that it preserves meets and joins.
	
	It is immediate from the definition of $\alpha$ that if $B\in \NC(\ZZ)$, under the identification of $\NC(\ZZ)$ with the partitions containing the singleton block $\{-\infty\}$, then $\alpha(B)\in \Sat(\A_f)$.
\end{proof}

\begin{lem}\label{lem_phi}
The assignment $\phi$ defines a map of lattices
\begin{displaymath}
\phi\colon \Sat(\A) \to \NC(\mathcal{Z})
\end{displaymath}
which sends $\Sat(\A_f)$ to $\NC(\ZZ)$.
\end{lem}
\begin{proof}
Let $\mcS$ be a saturated set of arcs and $\phi(\mcS)$ the collection of subsets of $\mathcal{Z}$ defined above. For every $a\in \mathcal{Z}$ we have $a\in \phi(\mcS)(a)$ so the $\phi(\mcS)(a)$ certainly cover $\mathcal{Z}$. Let us verify it is a partition. To slightly ease the notational load let us denote the partition $\phi(\mcS)$ simply by $B$.

	We claim that $b \in B(a)$ if and only if $B(a) = B(b)$. It is clear that $B(a) = B(b)$ implies $b \in B(a)$. So assume that $b \in B(a)$. If $a = b$, for instance if $B(a)$ is a singleton, then the statement is trivially true. If they are distinct then either $(a,b)$ or $(b,a)$ lies in $\mathcal{S}$. Let $c \in B(a)$, $d \in B(b)$. Then those of $(d,b), (b,d), (a,c)$ and $(c,a)$ which are arcs lie in $\mathcal{S}$ and since $\mathcal{S}$ is saturated, those of $(d,a), (a,d), (b,c)$ and $(c,b)$ which are arcs also lie in $\mathcal{S}$. It follows that $d \in B(a)$ and $c \in B(b)$ and therefore $B(a) = B(b)$.
	
Thus $\phi(\mcS)$ is a partition as it follows that distinct blocks have empty intersection. It remains to show that it is non-crossing. To this end suppose we are given $a,b \in B(a)$ and $c,d \in B(c)$ with $a < c < b < d$. Our assumption tells us that $(a,b) \in \mathcal{S}$ and $(c,d) \in \mathcal{S}$ and $(a,b)$ and $(c,d)$ cross. Since $\mathcal{S}$ is saturated, the arcs $(a,c), (c,b), (b,d)$ and $(a,d)$ also lie in $\mathcal{S}$. Hence $c\in B(a)$ and, using again the fact we proved above, we must have $B(a) = B(c)$.

If $\mcS\in \Sat(\A_f)$ then it contains no arc starting at $\{-\infty\}$ and so gives a non-crossing partition of $\NC(\ZZ)$ under our fixed embedding in $\NC(\mathcal{Z})$.

Finally, it's again evident that $\phi$ is a map of posets and a straightforward verification to see it respects the lattice structures.
\end{proof}

\begin{thm}
The maps $\alpha$ and $\phi$ are inverse isomorphisms of lattices $\NC(\mathcal{Z}) \cong \Sat(\A)$ which restrict to give an isomorphism $\NC(\ZZ)\cong \Sat(\A_f)$.
\end{thm}
\begin{proof}
By the preceding two lemmas we know $\alpha$ and $\phi$ are well defined morphisms of lattices and we just need to check they are inverse. So suppose we are given a non-crossing partition $B = \{B_i\;\vert\;i\in I\}$ and a saturated set of arcs $\mcS$. Then for $a\in B_j$ the partition $\phi\alpha(B)$ has a corresponding block
\begin{align*}
\phi\alpha(B)(a) &=  \{a\} \sqcup \{b\in \mathcal{Z} \; \vert \; (a,b) \; \text{or} \; (b,a) \; \text{lies in}\; \alpha(B)\} \\
&= \{a\} \sqcup \{b\in \mathcal{Z} \; \vert \; (a,b) \; \text{or} \; (b,a) \; \text{lies in}\; \bigcup_{i \in I}\{(a,b) \in \A \mid a,b \in B_i\}\} \\
&= B_j
\end{align*}
so $\phi\alpha(B) = B$. On the other hand, 
\begin{align*}
\alpha\phi(\mcS) &= \bigcup_{c\in \mathcal{Z}}\{(a,b) \in \A \mid a,b\in \phi(\mcS)(c)\} \\
&= \bigcup_{c\in \mathcal{Z}}\{(a,b) \in \A \mid a,b\in \{c\} \sqcup \{d\in \mathcal{Z} \; \vert \; (c,d) \; \text{or} \; (d,c) \; \text{lies in}\; \mcS\}\} \\
&= \mcS
\end{align*}
and so $\alpha$ and $\phi$ are inverse as claimed. 
\end{proof}

Summarising what we have shown yields the following corollary.

\begin{cor}\label{cor_summary}
There is a commutative diagram
\begin{displaymath}
\xymatrix{
\NC(\ZZ\sqcup\{-\infty\}) \ar[r]^-\sim & \Sat(\A) \ar[r]^-\sim & \Thick(\sfD^\mathrm{b}(\gr S)) \ar[r]^-\sim & \Thick(\sfD^\mathrm{b}(\gr R)) \\
\NC(\ZZ) \ar[u] \ar[r]_-\sim & \Sat(\A_f) \ar[u] \ar[r]_-\sim & \Thick(\sfD^\mathrm{b}_\mathrm{tors}(\gr S)) \ar[u] \ar[r]_-\sim & \Thick(\sfD^\mathrm{perf}(\gr R)) \ar[u] \\
}
\end{displaymath}
where the horizontal maps are lattice isomorphisms and the vertical maps are inclusions of posets preserving arbitrary meets and joins.
\end{cor}

%-------------------------------------------------------------------------------------------------------------------------------------------------

\subsection{Actions of autoequivalences}

To round out the discussion we now describe the action of some important autoequivalences of $\sfD^\mathrm{b}(\gr S)$ on the geometric model and the lattice of thick subcategories. It turns out that very natural operations on arcs, namely translations and reflections, correspond to similarly natural autoequivalences.

As noted in the preliminaries there is a grading shift, or twist, $(i)$ on graded modules for each $i\in \ZZ$. This extends canonically to an autoequivalence of $\sfD^\mathrm{b}(\gr S)$ which preserves the full subcategory $\sfD^\mathrm{b}_\mathrm{tors}(\gr S)$.

\begin{prop}\label{prop_translate}
The action of $(i)$ on $\sfD^\mathrm{b}(\gr S)$ corresponds to translation by $i$ on $\A$ i.e.\ it acts by
\begin{displaymath}
\begin{array}{lrlr}
& (a,b) &\mapsto &(a+1, b+1) \\
\text{and} & (-\infty, b) &\mapsto &(-\infty, b+1)
\end{array}
\end{displaymath}
The only thick subcategories stable under $(1)$ are 
\begin{displaymath}
0 \subsetneq \sfD^\mathrm{b}_\mathrm{tors}(\gr S) \subsetneq \sfD^\mathrm{b}(\gr S).
\end{displaymath}
\end{prop}
\begin{proof}
One easily sees that the action is as stated using the bijection between arcs and indecomposable objects. The second statement follows from Theorem~\ref{thm_saturatedclassification}. Suppose $\mcS$ is a non-empty saturated subset of arcs closed under translations. For any finite $(a,b)\in \mcS$ we obtain $(a+1, b+1)\in \mcS$ and hence $(a,a+1)\in \mcS$. Thus $\A_f \subseteq \mcS$ as soon as $\mcS$ contains a finite arc.

If $\mcS$ contains an infinite arc $(-\infty, b)$ then, in a similar manner we get $(b,b+1) \in \mcS$. Thus there are three possibilities for a translation closed saturated subset: the empty set $\varnothing$, the finite arcs $\A_f$, or all of $\A$.
\end{proof}

\begin{rem}
Closure under $(i)$ corresponds to being a thick tensor ideal of $\sfD^\mathrm{b}(\gr S)$ and so the above gives the expected answer according to the computation performed in \cite{DSgraded}. The three stable thick subcategories correspond to the three closed subsets of the homogeneous spectrum, $\{0, (y)\}$, of $S$.
\end{rem}

One can also consider reflections on $\A$ which give rise to anti-involutions on $\Sat(\A)$. Explicitly, if we fix $j\in \ZZ$ there is a reflection about $j$, denoted $s_j$, defined for $a,b\in \ZZ$ by
\begin{displaymath}
s_j(a,b) = (j-b, j-a) \;\; \text{and} \;\; s_j(-\infty, b) = (-\infty, j-b).
\end{displaymath}
One checks easily that arbitrary reflections can be expressed by the reflection $s_0$ about $0$ and translation via the relation $s_j = (j)\circ s_0$. These reflections correspond, up to suspensions, to Grothendieck duality on $\sfD^\mathrm{b}(\gr S)$. Let $\RuHom(-,-)$ denote the right derived graded hom-functor on $\sfD^\mathrm{b}(\gr S)$.

\begin{prop}\label{prop_reflect}
Under the bijection between suspension orbits of indecomposable objects of $\sfD^\mathrm{b}(\gr S)$ and arcs $\A$ the duality $\RuHom(-, S(j))$ corresponds to $s_j$. Thus the Grothendieck duality functor $\RuHom(-,S(j))$ acts on $\Thick(\sfD^\mathrm{b}(\gr S))\cong \Sat(\A)$ via reflection about $j$. In particular the thick subcategories stable under this duality are precisely those which correspond to saturated sets of arcs which are symmetric about $j$.
\end{prop}
\begin{proof}
One computes directly that
\begin{align*}
\RuHom(S/(y^i)(l), S(j)) &\cong \RuHom(S/(y^i), S)(j-l) \\
&\cong \Sigma S/(y^i)(i+j-l)
\end{align*}
which corresponds to the operation 
\begin{displaymath}
(l-i, l) \mapsto (j-l, j+i-l) = s_j(l-i, l)
\end{displaymath}
on the arc labeling this module. The verification for the $S(l)$ is similar.
\end{proof}

%-------------------------------------------------------------------------------------------------------------------------------------------------

%-------------------------------------------------------------------------------------------------------------------------------------------------

\section{Some comments on localising subcategories}\label{sec_loc}

We now make a few, relatively brief, remarks on the situation for $\sfD(\Gr R)$ the unbounded derived category of all graded $R$-modules. As this section is intended partially as speculation and partially as motivation for others to pursue this problem we are a bit light on details and assume some level of expertise from the reader.

For the unbounded derived category one would like to understand the lattice of localising subcategories, i.e.\ the triangulated subcategories closed under arbitrary coproducts. We do not have a classification of the localising subcategories in this case; in fact, we do not even have a proof there are a set of localising subcategories rather than a proper class (although the strong expectation is that there is a set and that this classification problem is tractable). 

In this situation one can still use Koszul duality to express the problem as asking for a classification of localising subcategories of $\sfD_\mathrm{tors}(\Gr S)$, the unbounded derived category of torsion graded $S$-modules. It would of course be interesting to have a classification for the whole unbounded derived category $\sfD(\Gr S)$ of $S$ which would correspond to understanding $\sfK(\Inj R)$, the homotopy category of complexes of injective $R$-modules, on the exterior algebra side.

As an advertisement let us say that, besides being a very natural question, a classification of localising subcategories for $R$ would give a classification of localising subcategories for $\sfD(\Gr A[x]/(x^2))$ where $A$ is any commutative noetherian ring. This would follow from work of the second author and Antieau \cite{StevensonAntieau}. Presumably in this case one could, in a conceptually satisfying way, extend Corollary~\ref{cor_summary} to $\sfD^\mathrm{b}(\gr A[x]/(x^2))$ and one would expect a classification via poset maps from $\Spec A$ to $\NC(\ZZ\sqcup\{-\infty\})$.

\subsection{Preliminary remarks}

Our discussion is centred around contrasting this case with the situation for $\sfD(\Modu \base A_n)$, the unbounded derived category of a finite A-type Dynkin quiver. As noted in the introduction, $\sfD(\Gr R)$ and $\sfD(\Gr S)$ can be viewed as infinite generalisations of the derived categories of these finite Dynkin quivers.

For $A_n$ the classification of localising subcategories is the same as the classification of thick subcategories of the compact objects $\sfD^\mathrm{b}(\modu \base A_n)$. This is because $\sfD(\Modu \base A_n)$ is \emph{pure semisimple}: every object is a direct sum of indecomposable compact objects. The category $\sfD(\Gr R)$ is certainly not pure semisimple as $\base$ is indecomposable but not compact. However, it is at first glance conceivable that, nonetheless, every object is a direct sum of indecomposable objects. In the first section below we give an argument that this is not the case. This result is not surprising and probably known to experts, but we include an argument as it was not immediately obvious to us and we could not locate a reference. 

In the final section we show that, although not every object is a sum of indecomposable objects, the objects of $\sfD(\Gr R)$ tend to have many indecomposable summands. As, by \cite{ALPP}, the indecomposable objects of $\sfD(\Gr R)$ are precisely those in $\sfD^\mathrm{b}(\gr R)$ this provides some hope that one can deduce a classification by exploiting Corollary~\ref{cor_summary}. 

\subsection{There are not enough indecomposables}

We will show that not every object of $\sfD(\Gr R)$ is a direct sum of indecomposable objects. The argument is purely formal and the main fact we will use is the following lemma which is well known to the initiated.

\begin{lem}
Let $\sfA$ be a Grothendieck category such that every injective object of $\sfA$ is a direct sum of indecomposable objects. Then direct sums of injectives in $\sfA$ are injective. In particular, the category $\sfA$ is locally noetherian.
\end{lem}
\begin{proof}
Let $\Lambda$ be an indexing set, let $\{J_\lambda\; \vert \; \lambda\in \Lambda\}$ be a set of injectives in $\sfA$, and let
\begin{displaymath}
F = E(\bigoplus_{\lambda \in \Lambda} J_\lambda)
\end{displaymath}
be the injective envelope of the direct sum of the $J_\lambda$. By hypothesis we can write
\begin{displaymath}
F = \bigoplus_{\gamma \in \Gamma} F_\gamma
\end{displaymath}
for some indexing set $\Gamma$, where each $F_\gamma$ is indecomposable and, as a summand of $F$, necessarily injective. By construction the inclusion
\begin{displaymath}
\bigoplus_\lambda J_\lambda \to F = \bigoplus_\gamma F_\gamma
\end{displaymath}
is essential and so, for all $\gamma\in \Gamma$, we have
\begin{displaymath}
(\bigoplus_\lambda J_\lambda) \cap F_\gamma \neq 0.
\end{displaymath}
So, fixing a $\gamma$, we can find $\lambda_1,\ldots,\lambda_n$ with
\begin{displaymath}
X = (F_\gamma \cap \bigoplus_{i=1}^n J_{\lambda_i}) \neq 0.
\end{displaymath}
But $F_\gamma$ is an indecomposable injective and so must be the injective envelope of $X$. From this we deduce that $F_\gamma$ is a summand of $\bigoplus_{i=1}^n J_{\lambda_i}$. Thus for each $\gamma$ we see $F_\gamma \subseteq \bigoplus_\lambda J_\lambda$. It follows immediately that
\begin{displaymath}
F = \bigoplus_\lambda J_\lambda
\end{displaymath}
and, in particular, the direct sum of the $J_\lambda$ is injective.

We have thus shown that every direct sum of injectives is injective. It is standard that this is equivalent to $\sfA$ being locally noetherian.
\end{proof}

Let us now explain how it follows from this lemma that not every object of $\sfD(\Gr R)$ can be a direct sum of indecomposable objects. By the graded analogue of the results of \cite{ALPP} the indecomposable objects of $\sfD(\Gr R)$ are all pure injective. So, if this were the case, then every pure injective object in $\sfD(\Gr R)$ would be a direct sum of indecomposable pure injective objects. Equivalently, this would say that every injective object in $\Modu \sfD^\mathrm{perf}(\gr R)$, the category of additive contravariant functors from the perfect complexes to abelian groups, was a direct sum of indecomposable objects\textemdash{}indecomposability is preserved since the restricted Yoneda functor is fully faithful on pure injective objects. By the lemma this would make $\Modu \sfD^\mathrm{perf}(\gr R)$ locally noetherian which is, in turn, equivalent to pure semisimplicity of $\sfD(\Gr R)$ \cite{KrTele}*{Theorem~2.10}. On the other hand $\base \in \sfD(\Gr R)$ is indecomposable and not compact so this conclusion is absurd.

\begin{rem}
The argument above is completely general and shows that if $\sfT$ is a compactly generated triangulated category in which every object is a direct sum of indecomposable pure injectives then $\sfT$ is pure semisimple. In fact, one can even remove the pure injectivity hypothesis provided passing to modules over $\sfT^c$ preserves indecomposability.
\end{rem}

\subsection{There are many indecomposable summands}

Despite the result of the last section it turns out that ``large'' objects of $\sfD(\Gr R)$ generally have many indecomposable summands. In fact, any cohomology class of an object of $\sfD(\Gr R)$ can be realised by a summand lying in $\sfD^\mathrm{b}(\gr R)$. Although the category $\sfD(\Gr R)$ is very special this is, nonetheless, an amusing fact and as alluded to earlier could be useful in understanding the lattice of localising subcategories, which would in turn shed light on general order $2$ nilpotent thickenings.

The argument involves a lot of technicalities as it is based on constructing families of compatible splittings. Since this is somewhat tangential to the main points of the paper we just provide a few observations that could be useful in treating the localising subcategories by reducing to considering bounded complexes to at least some extent.

Given a cochain complex $E$ we denote by $Z^j(E)$ and $B^j(E)$ the $j$-cocycles and $j$-coboundaries, i.e.\ $\ker d^j$ and $\im d^{j-1}$, respectively.  

The first step is to remove free summands from the cohomology. This is relatively straightforward since $R$ is self-injective.

\begin{lem}\label{lem_freesplitpayoff}
Let $E$ be a complex of graded $R$-modules. Then in $\sfD(\Gr R)$ we can decompose $E$ as $E' \oplus E''$ with $E' \in\Add(\Sigma^jR(i)\;\vert\; i,j\in \ZZ)$ and $E''$ having no free summands in its cohomology i.e.\ $H^j(E'')\in \Add(\base(i)\;\vert\; i\in \ZZ)$ for each $j\in \ZZ$.
\end{lem}

We can thus reduce to considering complexes whose cohomology groups are semisimple. 

Recall that $E$ is a homotopically minimal complex of injectives if each $Z^j(E)\to E^j$ is an injective envelope and that any complex is quasi-isomorphic to such a complex; this can be seen by taking a K-injective resolution of $E$ and then applying \cite{KrStab}*{Proposition~B.2}. So from this point onward we may as well fix a homotopically minimal complex of injectives $E$ whose cohomology groups all lie in $\Add(\base(j)\;\vert\; j\in \ZZ)$.

The next step is to show such an $E$ is also a homotopically minimal complex of projectives, i.e.\ each of the canonical morphisms
\begin{displaymath}
F^i \to \coker d^{i-1}
\end{displaymath}
is a projective cover. This is accomplished by further analysing the cocycles and coboundaries, the key points being given by the following technical lemma.

\begin{lem}\label{lem_socle}
Suppose we are given a monomorphism $\alpha\colon R(j)\to Z^i(E)$. Then
\begin{displaymath}
\alpha(R(j))\cap B^i(E) = \alpha((x)) = \soc(\alpha(R(j))).
\end{displaymath}
In particular, $B^i(E)\in \Add(\base(j)\;\vert \;j\in \ZZ)$ for all $i\in \ZZ$. 
\end{lem}
\begin{proof}
As $R(j)$ is injective we can split $\alpha$ and the further inclusion $\alpha'\colon R(j)\to E^i$ compatibly by considering the diagram
\begin{displaymath}
\xymatrix{
R(j) \ar[d]_{1_{R(j)}} \ar[r]^-\alpha & Z^i(E) \ar@{-->}[dl]_-\beta \ar[r] & E^i \ar@{-->}[dll]^-{\beta'} \\
R(j) & &
}
\end{displaymath}
One then checks, exploiting self-injectivity of $R$, that $\alpha(R(j)) \subseteq B^i(E)$ would contradict minimality of $E$ and $\alpha(R(j)) \cap B^i(E) = 0$ would contradict the cohomology being semisimple.
\end{proof}

\begin{lem}\label{lem_homin}
The complex $E$ is a homotopically minimal complex of projectives.
\end{lem}
\begin{proof}
Using the previous lemma we can write
\begin{displaymath}
Z^i(E) = F\oplus T \oplus T' \quad \text{and} \quad B^i(E) = S\oplus S',
\end{displaymath}
where $F$ is free and $T, T', S,$ and $S'$ are in $\Add(\base(j)\;\vert\; j\in \ZZ)$, such that the inclusion $B^i(E)\to Z^i(E)$ has the form
\begin{displaymath}
\xymatrix{
S\oplus S' \ar[rr]^-{\begin{pmatrix}x & 0 & 0 \\ 0 & 1 & 0 \end{pmatrix}} && F\oplus T \oplus T'
}
\end{displaymath}
Using this description and the fact that $E$ is a homotopically minimal complex of injectives we can write $E^i \to \coker d^{i-1}$ as
\begin{displaymath}
E^i = F \oplus E(T) \oplus E(T') \to F/xF \oplus E(T)/xE(T) \oplus E(T')
\end{displaymath}
where $E(-)$ denotes the injective envelope and the map is the obvious projection. This is visibly a projective cover and so $E$ is a homotopically minimal complex of projectives as claimed. 
\end{proof}

One can now start producing summands of $E$. This is done in two cases, depending on whether we need to start building a subcomplex of $E$ to the left or to the right of our starting point. Let us assume $H^i(E) \neq 0$ and choose an embedding $\alpha\colon \base \to H^i(E)$. It is sufficient to treat the case of $\base$ living in degree $0$ as one can just apply the grading shift $(i)$ to deal with copies of $\base(i)$.

\subsubsection{First case}\label{sec_1st}
Suppose there is a $\gamma^i\colon \base\to Z^i(E)$ making the diagram
\begin{displaymath}
\xymatrix{
\base \ar[d]_-{\gamma^i} \ar[dr]^-\alpha & \\
Z^i(E) \ar[r] & H^i(E)
}
\end{displaymath}
commute, i.e.\ the image of $\alpha$ is realised by a copy of $\base$ in the cocycles. As $E$ is minimal there is a unique summand $R(1)$ of $E^i$, given by a map $\widetilde{\gamma}^i\colon R(1) \to E^i$, into which $\gamma^i(\base)$ embeds. We can complete this to a commutative diagram
\begin{displaymath}
\xymatrix{
0 \ar[r] \ar[d] & \base \ar[d]^-{\gamma^i} \ar[r] & R(1) \ar[d]^-{\widetilde{\gamma}^i} \ar[r] & \base(1) \ar[d]^-{\exists!\;\gamma^{i+1}} \\
E^{i-1} \ar[r]_-d & E^i \ar@{=}[r] & E^i \ar[r]_-d & E^{i+1}
}
\end{displaymath}
where the rightmost map is just the factorisation of $d\widetilde{\gamma}^i$ via its image. Again using that $E$ is minimal we can produce, essentially uniquely, a further commutative square
\begin{displaymath}
\xymatrix{
\base(1) \ar[r] \ar[d]_-{\gamma^{i+1}} & R(2) \ar[d]^-{\widetilde{\gamma}^{i+1}} \\
E^{i+1} \ar@{=}[r] & E^{i+1}
}
\end{displaymath}
by factoring $\gamma^{i+1}$ via an injective envelope. Noting that the composite
\begin{displaymath}
R(1) \to \base(1) \to R(2)
\end{displaymath}
is just multiplication by $x$ all this taken together gives us a commutative diagram
\begin{equation}\label{eq_diag1}
\xymatrix{
0 \ar[r] \ar[d] & \ar[r] R(1) \ar[r]^-x \ar[d]_-{\widetilde{\gamma}^i} & R(2) \ar[d]^-{\widetilde{\gamma}^{i+1}} \\
E^{i-1} \ar[r]_-d & E^i \ar[r]_-d & E^{i+1}
}
\end{equation}

There are now two possibilities: either $\widetilde{\gamma}^{i+1}R(2)$ lies in $Z^{i+1}(E)$, in which case we are done, or it does not. If $\widetilde{\gamma}^{i+1}R(2)$ does not lie in $Z^{i+1}(E)$ then we can iterate the above procedure, producing monomorphisms
\begin{displaymath}
\widetilde{\gamma}^{i+n}\colon R(n+1)\to E^{i+n}
\end{displaymath}
either indefinitely or until $\widetilde{\gamma}^{i+n}R(n+1)$ lies in $Z^{i+n}(E)$. In this way we produce a complex $E_\gamma$ which is the unique indecomposable perfect complex with the given length if the process terminates at some finite stage or a minimal injective resolution of $\base$ otherwise. In either case it is certainly in $\sfD^\mathrm{b}(\gr R)$ and we have the following lemma.

\begin{lem}\label{lem_splitter}
The map $\widetilde{\gamma}\colon E_\gamma \to E$ given by assembling the $\widetilde{\gamma}^j$ exhibits $E_\gamma$ as a subcomplex of $E$. This inclusion admits a retraction which can be constructed explicitly. Moreover, $E_\gamma$ is an indecomposable complex in $\sfD^\mathrm{b}(\gr R)$ and the morphism $H^i(\widetilde{\gamma})$ is none other than the map $\alpha$ we started with.
\end{lem}
\begin{proof}
By construction we have $d\widetilde{\gamma}^j = \widetilde{\gamma}^{j+1}d$ (see for instance (\ref{eq_diag1})) and so $E_\gamma$ is a subcomplex of $E$ as claimed. The complex $E_\gamma$ is evidently an indecomposable bounded complex of finitely generated free modules and at the beginning of the story we picked $\widetilde{\gamma}^i$ precisely so that it induced $\alpha$ on cohomology. The real work is in producing a section for $\widetilde{\gamma}$. This can be accomplished by producing sections step by step during the process of producing $E_\gamma$, a task which requires a fair amount of artifice and is somewhat tedious.
\end{proof}

\subsubsection{Second case}\label{sec_2nd}
Again let us assume $H^i(E) \neq 0$ and choose an embedding $\alpha = \alpha^i\colon \base \to H^i(E)$.

If we are not in the first case we must have a monomorphism, which is necessarily split, $\widetilde{\alpha}^i\colon R\to Z^i(E)$ making the following square commute
\begin{displaymath}
\xymatrix{
R \ar[r] \ar[d]_-{\widetilde{\alpha}^i} & \base\ar[d]^-{\alpha^i} \\
Z^i(E) \ar[r] & H^i(E)
}
\end{displaymath}
We proceed by considering the commutative diagram
\begin{displaymath}
\xymatrix{
R(-1) \ar[dr] \ar@{-->}[ddr]_-{\widetilde{\alpha}^{i-1}} \ar[r] & \base(-1) \ar[d]^-{\alpha^{i-1}} \ar[r] & R \ar[r] \ar[d]^-{\widetilde{\alpha}^i}  & \base\ar[d]^-{\alpha^i} \\
& B^i(E) \ar[r] & Z^i(E) \ar[d] \ar[r] & H^i(E) \\
& E^{i-1} \ar[u] \ar[r]_-{d^{i-1}} & E^i
}
\end{displaymath}
where the dashed arrow exists by projectivity of $R(-1)$ and $\widetilde{\alpha}^{i-1}$ can be chosen to be a monomorphism. Again we can iterate, constructing a subcomplex $E_\alpha$ in increasingly negative degrees together with compatible retractions, either indefinitely or until we hit an $R(-n)\in Z^{i-n}(E)$. This gives the ``dual'' version of Lemma~\ref{lem_splitter} (``dual'' in the sense that the construction is dual, the statement doesn't change much):

\begin{lem}\label{lem_splitter2}
The map $\widetilde{\alpha}\colon E_\alpha \to E$ given by assembling the $\widetilde{\alpha}^j$ exhibits $E_\alpha$ as a subcomplex of $E$. This inclusion admits a retraction which can be constructed explicitly. Moreover, $E_\alpha$ is an indecomposable complex in $\sfD^\mathrm{b}(\gr R)$ and the morphism $H^i(\widetilde{\alpha})$ is none other than the map $\alpha$ we started with.
\end{lem}

%-------------------------------------------------------------------------------------------------------------------------------------------------

%-------------------------------------------------------------------------------------------------------------------------------------------------

\begin{ack}
We are grateful to Adam-Christiaan van Roosmalen and to Jan Stovicek for interesting discussions on this and related subjects; Adam-Christiaan in particular was very generous in sharing his thoughts and has informed us that these ideas can be generalised to treat `larger' type $A$ examples. We also thank Henning Krause for encouraging us to pursue this project.
\end{ack}

\bibliography{greg_bib}

\end{document}